\documentclass{article}%
\usepackage{amsmath}
\usepackage{amsfonts}
\usepackage{amssymb}
\usepackage{graphicx}%
\providecommand{\U}[1]{\protect \rule{.1in}{.1in}}
\newtheorem{theorem}{Theorem}

\newtheorem{definition}[theorem]{Definition}

\newtheorem{lemma}[theorem]{Lemma}

\newtheorem{remark}[theorem]{Remark}

\newenvironment{proof}[1][Proof]{\noindent \textbf{#1.} }{\  \rule{0.5em}{0.5em}}
\begin{document}

\title{Tightness, weak compactness of nonlinear expectations and application to CLT }
\author{Shige PENG\thanks{Partially supported by National Basic Research Program of
China (973 Program) (No. 2007CB814906), and NSF of China (No. 10921101). The
author thanks to some discussions with Zengjing Chen, Fuqing Gao, Mingshang
Hu, Xinpeng Li, Zhengyan Lin, Qiman Shao and Liming Wu. It is those
discussions which motivated the author to the realization of the present
paper.}\\School of Mathematics, Shandong University\\250100, Jinan, China}
\date{version: June 10, 2010}

\maketitle

\medskip\medskip\medskip

\noindent\textbf{Abstract.} In this paper we introduce a notion of
tightness for a family of nonlinear expectations and show that the
tightness can be applied to obtain weak compactness in a framework
of nonlinear expectation space. This criterion is very useful for
obtaining the weak convergence for a sequence of nonlinear
expectations, which is a equivalent to the so-called convergence in
distribution, or in law for a sequence of random variables in a
nonlinear expectation space. We use the above result to give a new
proof to the central limit theorem under a sublinear expectation
space. The method can be also applied to prove the convergence of
some numerical schemes for degenerate fully nonlinear PDEs.

\medskip\medskip\medskip

\section{Introduction}

In this paper we introduce a notion of tightness for a family of
nonlinear expectations and show that the tightness can be applied to
obtain weak compactness in a framework of nonlinear expectation
space. This criterion is very useful for obtaining the weak
convergence for a sequence of nonlinear expectations, which is a
equivalent to the so-called convergence in distribution, or in law
for a sequence of random variables in a nonlinear expectation space.
Our results have nontrivially generalized those in the classical
situation of weak convergence, or convergence in law in classical
probability theory.

As an application of those results, we provide a new and probabilistic proof
for central limit theorem under sublinear expectation. Such new type of CLT
was obtained in [Peng2007b], [Peng2008] (see also [Peng2010] for a more
general situation). But the proof given in the above papers depend heavily on
a very deep estimate of fully nonlinear parabolic PDE of parabolic types. Our
new proof also provides a probabilistic method to prove the existence of fully
nonlinear parabolic PDEs. The method can be also applied to prove the
convergence of some numerical schemes for degenerate nonlinear PDEs.

This paper is organized as follows, in the next section we give a
preliminary introduction for the notion of sublinear and nonlinear
expectations. New notions of tightness and weak convergence of
nonlinear expectations and their relations will be introduced in
Section 3. In Section 4 we will use the result of the previous
section to give a new and probabilistic proof of the central limit
theorem under sublinear expectations.

\section{Preliminaries: sublinear expectation and upper expectation}

We present some basic notions and results of nonlinear and sublinear
expectations. For more details we refer to [Peng2007a], [Denis,Hu,Peng2008],
[Peng2010] Let $\Omega$ be a given set and $\mathcal{H}$ be a linear space of
real functions on $\Omega$ such that $1\in \mathcal{H}$ and $|\xi
|\in \mathcal{H}$ if $\xi \in \mathcal{H}$. A functional $\mathbb{E}%
:\mathcal{H}\mapsto \mathbb{R}$ is called a nonlinear expectation if it
satisfies $\mathbb{E}[c]=c$, $c\in \mathbb{R}$ and $\mathbb{E}[X]\geq
\mathbb{E}[Y]$, for $X$, $Y\in \mathcal{H}$ such that $X\geq Y$. If
$\mathbb{E}$ satisfies furthermore $\mathbb{E}[X+Y]-\mathbb{E}[X]-\mathbb{E}%
[Y]\leq0$ (resp. $=0$) and $\mathbb{E}[\lambda X]=\lambda \mathbb{E}[X]$, for
$X$, $Y\in \mathcal{H}$ and $\lambda \geq0$, then we call it a sublinear (resp.
linear) expectation.

Recently a new notion of coherent risk measures in finance caused much
attention to the study of such type of sublinear expectations and applications
to risk controls, see the seminal paper of Artzner, Delbaen, Eber and Heath
(1999) as well as in F\"{o}llmer and Schied (2004).

The following result is well-known as representation theorem. It is a direct
consequence of Hahn-Banach theorem (see Delbaen (2002), F\"{o}llmer and Schied (2004)).

\begin{theorem}
\bigskip Let {{$\mathbb{\hat{E}}$}} be a sublinear expectation defined on
$(\Omega,\mathcal{H})$. Then there exists a family of linear expectations
$\{E_{\theta}:\theta \in \Theta \}$ on $(\Omega,\mathcal{H})$ such that
\[
{{\mathbb{\hat{E}}}}[X]=\max_{\theta \in \Theta}E_{\theta}[X].
\]

\end{theorem}

A sublinear expectation $\mathbb{\hat{E}}$ on $(\Omega,\mathcal{H})$ is said
to be regular if for each sequence $\{X_{n}\}_{n=1}^{\infty}\subset
\mathcal{H}$ such that $X_{n}(\omega)\downarrow0$, for $\omega$, we have
${{\mathbb{\hat{E}}}}[X_{n}]\downarrow0$. If $\mathbb{\hat{E}}$ is regular
then from the above representation we have $E_{\theta}[X_{n}]\downarrow0$ for
each $\theta \in \Theta$. It follows from Daniell-Stone theorem that there
exists a unique ($\sigma$-additive) probability measure $P_{\theta}$ defined
on $(\Omega,\sigma(\mathcal{H}))$ such that
\[
E_{\theta}[X]=\int_{\Omega}X(\omega)dP_{\theta}(\omega),\  \ X\in
\mathcal{H}\text{.}%
\]

The above representation theorem of sublinear expectation tells us that to use
a sublinear expectation for a risky loss $X$ is equivalent to take the upper
expectation of $\{E_{\theta}:\theta \in \Theta \}$. The corresponding model
uncertainty of probabilities, or ambiguity, is the subset $\{P_{\theta}%
:\theta \in \Theta \}$. The corresponding uncertainty of distributions for an
$n$-dimensional random variable $X$ in $\mathcal{H}$ is
\[
\{F_{X}(\theta,A):=P_{\theta}(X\in A):A\in \mathcal{B}(\mathbb{R}^{n})\}.
\]

We consider the case where $\Omega$ is a complete separable metric space. Let
$\mathcal{B}(\Omega)$ be the Borel $\sigma$-algebra of $\Omega$ and
$\mathcal{M}$ the collection of all probability measures on $(\Omega
,\mathcal{B}(\Omega))$.

\begin{itemize}
\item $L^{0}(\Omega)$: the space of all $\mathcal{B}(\Omega)$-measurable real functions;

\item $B_{b}(\Omega)$: all bounded functions in $L^{0}(\Omega)$;

\item $C_{b}(\Omega)$: all continuous functions in $B_{b}(\Omega)$.
\end{itemize}

We are given a subset $\mathcal{P}\subseteq \mathcal{M}$. The upper expectation
of $\mathcal{P}$ is defined as follows: for each $X\in L^{0}(\Omega)$ such
that $E_{P}[X]$ exists for each $P\in \mathcal{P}$,
\[
\mathbb{E}[X]=\mathbb{E}^{\mathcal{P}}[X]:=\sup_{P\in \mathcal{P}}E_{P}[X].
\]

It is easy to verify that the upper expectation $\mathbb{E}[\cdot]$
of the family $\mathcal{P}$ is a sublinear expectation on
$(\Omega,B_{b}(\Omega))$ as well as on $(\Omega,C_{b}(\Omega)$).

Moreover, it is also easy to check that

\begin{theorem}
\begin{enumerate}
\item Let $\mathbb{E}[X_{n}]$ and $\mathbb{E}[\sum_{n=1}^{\infty}X_{n}]$ be
finite. Then $\mathbb{E}[\sum_{n=1}^{\infty}X_{n}] \leq \sum_{n=1}^{\infty}
\mathbb{E}[X_{n}].$

\item Let $X_{n} \uparrow X$ and $\mathbb{E}[X_{n}]$, $\mathbb{E}[X]$ be
finite. Then $\mathbb{E}[X_{n}] \uparrow \mathbb{E}[X]$.
\end{enumerate}
\end{theorem}

\begin{definition}
The functional $\mathbb{E}[\cdot]$ is said to be \textbf{regular} if for each
$\{X_{n}\}_{n=1}^{\infty}$ in $C_{b}(\Omega)$ such that $X_{n}\downarrow0$
\ on $\Omega$, we have $\mathbb{E}[X_{n}]\downarrow0$.
\end{definition}

We have:

\begin{theorem}
\label{Thm2}$\mathbb{\mathbb{E}}[\cdot]$ is regular if and only if
$\mathcal{P}$ is relatively compact.\
\end{theorem}

We denote
\[
c(A):=\sup_{P\in \mathcal{P}} P(A),\  \  \  \ A\in \mathcal{B}(\Omega).
\]

One can easily verify the following theorem.

$c(\cdot)$ is a Choquet capacity, i.e. (see [Choquet1953]),

\begin{theorem}
\begin{enumerate}
\item $0\leq c(A)\leq1,\  \  \forall A\subset \Omega$.

\item If $A\subset B$, then $c(A)\leq c(B)$.

\item If $(A_{n})_{n=1}^{\infty}$ is a sequence in $\mathcal{B}(\Omega)$, then
$c(\cup A_{n})\leq \sum c(A_{n})$.

\item If $(A_{n})_{n=1}^{\infty}$ is an increasing sequence in $\mathcal{B}%
(\Omega)$: $A_{n}\uparrow A=\cup A_{n}$, then $c(\cup A_{n})=\lim
_{n\rightarrow \infty}c(A_{n})$.
\end{enumerate}
\end{theorem}

Furthermore, for each $A\in \mathcal{B}(\Omega)$, we have
\[
c(A)=\sup \{c(K):\ K\text{ compact,}\ K\subset A\}.
\]
Indeed,
\[
c(A)=\sup_{P\in \mathcal{P}}\sup_{\substack{K\, \text{compact}\\K\subset
A}}P(K)=\sup_{\substack{K\, \text{compact}\\K\subset A}}\sup_{P\in \mathcal{P}%
}P(K)=\sup_{\substack{K\, \text{compact}\\K\subset A}}c(K).
\]
We use the standard capacity-related vocabulary: a set $A$ is \textbf{polar}
if $c(A)=0$ and a property holds \textquotedblleft \textbf{quasi-surely}%
\textquotedblright \ (q.s.) if it holds outside a polar set. In other words,
$A\in \mathcal{B}(\Omega)$ is polar if and only if $P(A)=0$ for any
$P\in \mathcal{P}$. We also have in a trivial way a Borel-Cantelli Lemma: Let
$(A_{n})_{n\in \mathbb{N}}$ be a sequence of Borel sets such that
\[
\sum_{n=1}^{\infty}c(A_{n})<\infty.
\]
Then $c(\limsup_{n\rightarrow \infty}A_{n})=0$. To prove it we just need to
applying the Borel-Cantelli Lemma under each probability $P\in \mathcal{P}$.

The following theorem is Prohorov's theorem.

\begin{theorem}
$\mathcal{P}$ is relatively compact if and only if for each $\varepsilon>0$,
there exists a compact set $K$ such that $c(K^{c})<\varepsilon$.
\end{theorem}

\section{Tightness and weak convergence of nonlinear expectations}

Let $\Omega$ be a complete separable metric space and let $\{ \mathbb{E}%
_{\alpha}\}_{\alpha \in \mathcal{A}}$ be a family of sublinear expectations on
$(\Omega,B_{b}(\Omega))$.

\begin{definition}
A sublinear expectation $\mathbb{\hat{E}}$ on $(\Omega,B_{b}(\Omega))$ is said
to be tight if for each $\varepsilon>0$, there exists a compact subset
$K\subset \Omega$ such that $\mathbb{\hat{E}}[1_{K^{c}}]<\varepsilon$. A family
of sublinear expectations $\{ \mathbb{E}_{\alpha}\}_{\alpha \in \mathcal{A}}$ on
$(\Omega,B_{b}(\Omega))$ is said to be tight if the sublinear expectation
defined by
\[
\mathbb{\hat{E}}[X]:=\sup_{\alpha \in \mathcal{A}}\mathbb{E}_{\alpha
}[X],\  \ X\in B_{b}(\Omega),
\]
is tight. A family of nonlinear expectations $\{ \mathbb{\tilde{E}}_{\alpha
}\}_{\alpha \in \mathcal{A}}$ on $(\Omega,C_{b}(\Omega))$ is said to be tight if
there exists a tight sublinear expectation $\mathbb{\hat{E}}$ on
$(\Omega,B_{b}(\Omega))$ such that, for each $\alpha \in \mathcal{A}$,
\[
\mathbb{\tilde{E}}_{\alpha}[X]-\mathbb{\tilde{E}}_{\alpha}[Y]\leq
\mathbb{\hat{E}}[X-Y],\  \ X,Y\in C_{b}(\Omega).
\]

\end{definition}

\begin{definition}
A sequence $\{ \mathbb{E}_{n}\}_{n=1}^{\infty}$ of nonlinear linear
expectations defined on $(\Omega,C_{b}(\Omega))$ is said to be weakly
convergent if, for each $\varphi \in C_{b}(\Omega)$, $\{ \mathbb{E}_{n}%
[\varphi]\}_{n=1}^{\infty}$ is a Cauchy sequence. A family $\{ \mathbb{E}%
_{\alpha}\}_{\alpha \in \mathcal{A}}$ of nonlinear linear expectations defined
on $(\Omega,C_{b}(\Omega))$ is said to be weakly compact if for each sequence
$\{ \mathbb{E}_{\alpha_{i}}\}_{i=1}^{\infty}$ there exists a weak convergent subsequence.
\end{definition}

\begin{theorem}
\label{Theorem1}Let a family of nonlinear expectations $\{ \mathbb{E}_{\alpha
}\}_{\alpha \in \mathcal{A}}$ be tight, namely there exists a tight sublinear
expectation $\mathbb{\hat{E}}$ on $(\Omega,B_{b}(\Omega)\mathbb{)}$ such that
\[
\mathbb{E}_{\alpha}[X]-\mathbb{E}_{\alpha}[Y]\leq \mathbb{\hat{E}%
}[X-Y],\  \ X,Y\in C_{b}(\Omega).
\]
Then $\{ \mathbb{E}_{\alpha}\}_{\alpha \in \mathcal{A}}$ is weakly compact,
namely, for each sequence $\{ \mathbb{E}_{\alpha_{n}}\}_{n=1}^{\infty}$, there
exists a subsequence $\{ \mathbb{E}_{\alpha_{n_{i}}}\}_{i=1}^{\infty}$ such
that, for each $\varphi \in C_{b}(\Omega)$, $\{ \mathbb{E}_{\alpha_{n_{i}}%
}[\varphi]\}_{i=1}^{\infty}$ is a Cauchy sequence.
\end{theorem}

\begin{proof}
Let $K_{i}$, $i=1,2,\cdots$, be a sequence of compact subsets in
$\Omega$ such that $K_{i}\subset K_{i+1}$ and
$\mathbb{\hat{E}}[\mathbf{1}_{(K_{i}})^c]\leq1/i$. Let $\{
\varphi_{j}\}_{j=1}^{\infty}$ constitute a linear subspace of
$C_{b}(\Omega)$ such that, for each $K_{i}$, $\{ \varphi_{j}(x)\}_{j=1}%
^{\infty}|_{x\in K}$ is dense in $C_{b}(K_{i})$. Since $\{ \mathbb{E}%
_{\alpha_{n}}[\varphi_{1}]\}_{n=1}^{\infty}$ is a bounded sequence, there
exists a subsequence $\{n_{i}^{1}\}_{i=1}^{\infty}$ of $n=1,2,\cdots$, such
that $\{ \mathbb{E}_{\alpha_{n_{i}^{1}}}[\varphi_{1}]\}_{i=1}^{\infty}$ is a
Cauchy sequence. Similarly, there exists a \ subsequence $\{n_{i}^{2}%
\}_{i=1}^{\infty}$ of $\{n_{i}^{1}\}_{i=1}^{\infty}$ such that $\{
\mathbb{E}_{\alpha_{n_{i}^{2}}}[\varphi_{2}]\}_{i=1}^{\infty}$ is a Cauchy
sequence. Proceeding in this way, we can find, for each $j=1,2,\cdots$, a
subsequence $\{n_{i}^{j}\}_{i=1}^{\infty}$ of $\{n_{i}^{j-1}\}_{i=1}^{\infty}$
such tat $\{ \mathbb{E}_{\alpha_{n_{i}^{j}}}[\varphi_{j}]\}_{i=1}^{\infty}$ is
a Cauchy sequence. It follows that, for the diagonal subsequence $\{n_{i}%
^{i}\}_{i=1}^{\infty}$ and for each fixed $j=1,2,\cdots$, $\{ \mathbb{E}%
_{\alpha_{n_{i}^{i}}}[\varphi_{j}]\}_{i=1}^{\infty}$ is a Cauchy sequence.

We now prove that, for each $\varphi \in C_{b}(\Omega)$, $\{ \mathbb{E}%
_{\alpha_{n_{i}^{i}}}[\varphi]\}_{i=1}^{\infty}$ is also a Cauchy
sequence. We denote $M=\sup_{\omega \in {\Omega}}|\varphi(\omega)|$.
For each $\varepsilon>0$, let $m$ be an integer larger than
$16M/\varepsilon$, thus
$\mathbb{\hat{E}}[\mathbf{1}_{(K_{m})^{c}}]\leq \varepsilon/16M$.
Let $\{
\varphi_{j_{l}}\}_{l=1}^{\infty}$ be a subsequence of $\{ \varphi_{j}%
\}_{j=1}^{\infty}$ such that $\sup_{l}\sup_{\omega \in \Omega}|\varphi_{i_{l}%
}(\omega)|\leq M$ and
\[
\lim_{l\rightarrow \infty}\sup_{\omega \in K_{m}}|\varphi_{i_{l}}(\omega
)-\varphi(\omega)|=0.
\]
Thus there exists an integer $l$ such that $\mathbb{\hat{E}}[|\varphi
-\varphi_{j_{l}}|\mathbf{1}_{K_{m}}]\leq \varepsilon/8$. It follows that
\begin{align*}
\mathbb{\hat{E}}[|\varphi-\varphi_{j}|]  &  \leq \mathbb{\hat{E}}%
[|\varphi-\varphi_{j_{l}}|\mathbf{1}_{(K_{m})^{c}}]+\mathbb{\hat{E}}%
[|\varphi-\varphi_{j_{l}}|\mathbf{1}_{K_{m}}]\\
&  \leq2M\varepsilon/16M+\mathbb{\hat{E}}[|\varphi-\varphi_{j_{l}}%
|\mathbf{1}_{K_{m}}]\\
&  \leq \varepsilon/4.\  \
\end{align*}
For this fixed $l$, there exists an integer $i_{0}$, such that $|\mathbb{E}%
_{\alpha_{i}^{i}}[\varphi_{j_{l}}]-\mathbb{E}_{\alpha_{j}^{j}}[\varphi_{j_{l}%
}]|\leq \varepsilon/2$, for each $i,j\geq i_{0}$. We then have
\begin{align*}
|\mathbb{E}_{\alpha_{n_{i}^{i}}}[\varphi]-\mathbb{E}_{\alpha_{n_{j}^{j}}%
}[\varphi]|  &  \leq \mathbb{\hat{E}}[|\varphi-\varphi_{j_{l}}|]+|\mathbb{E}%
_{\alpha_{n_{i}^{i}}}[\varphi_{j_{l}}]-\mathbb{E}_{\alpha_{n_{j}^{j}}}%
[\varphi_{j_{l}}]|+\mathbb{\hat{E}}[|\varphi-\varphi_{j_{l}}|]\\
&  \leq2\mathbb{\hat{E}}[|\varphi-\varphi_{j_{l}}|]+|\mathbb{E}_{\alpha
_{n_{i}^{i}}}[\varphi_{j_{l}}]-\mathbb{E}_{n_{j}^{j}}[\varphi_{j_{l}}%
]|\leq \varepsilon.\  \  \
\end{align*}
Thus $\{ \mathbb{E}_{\alpha_{n_{i}^{i}}}[\varphi]\}_{i=1}^{\infty}$ is a
Cauchy sequence. The proof is complete.
\end{proof}

\section{Convergence in distribution}

In this section we apply Theorem \ref{Theorem1} to give a probabilistic proof
of the central limit theorem under a sublinear expectation space.

\begin{definition}
Let $(\Omega,\mathcal{H},\mathbb{E})$ be a nonlinear expectation space and
$\left \{  \xi_{i}\right \}  _{i=1}^{\infty}$ be a sequence of $\mathbb{R}^{d}%
$-valued random variables in this space. $\left \{  \xi_{i}\right \}
_{i=1}^{\infty}$ is said to converge in distribution, or in law, if, for each
$\varphi \in C_{b}(\mathbb{R}^{d})$, $\{ \mathbb{E}[\varphi(\xi_{i}%
)]\}_{i=1}^{\infty}$ is a Cauchy sequence. We define
\[
\mathbb{F}_{\xi}[\varphi]=\lim_{i\rightarrow \infty}\mathbb{E}[\varphi(\xi
_{i})],\  \  \varphi \in C_{b}(\mathbb{R}^{d})\text{.}%
\]
$\mathbb{F}:C_{b}(\mathbb{R}^{d})\mapsto \mathbb{R}$ forms a nonlinear
expectation on $(\mathbb{R}^{d},C_{b}(\mathbb{R}^{d}))$.
\end{definition}

\begin{remark}
If $\mathbb{E}$ is a sublinear (resp. linear) expectation then $\mathbb{F}$ is
also a sublinear (resp. linear) on $(\mathbb{R}^{d},C_{b}(\mathbb{R}^{d}))$.
\end{remark}

\begin{lemma}
Let $(\bar{\Omega},\mathcal{\bar{H}},\mathbb{\bar{E}})$ be a sublinear
expectation space and let $\bar{X}=(\bar{X}_{1},\cdots,\bar{X}_{d})$ be such
that $\bar{X}_{1}$,$\cdots,\bar{X}_{d}\in \mathcal{\bar{H}}$ and $\psi(\bar
{X})\in \mathcal{\bar{H}}$ for each $\psi \in L_{0}(\mathbb{R}^{d})$ such that
$|\psi(x)|\leq c(1+|x|^{l})$, where $l>0$ is fixed. Then the sublinear
distribution on $(\mathbb{R}^{d},B_{b}(\mathbb{R}^{d}))$ defined by
$\mathbb{\bar{F}}_{\bar{X}}[\varphi]:=\mathbb{\bar{E}}[\varphi(\bar{X}%
)]:B_{b}(\mathbb{R}^{d})\mapsto \mathbb{R}$ is tight.
\end{lemma}

\begin{proof}
For each $\varepsilon>0$, we can choose an integer $N>0$ such that
$N^{l}>\mathbb{\bar{E}}[|\bar{X}|^{l}]/\varepsilon$ and a compact
set $K=\{x\in \mathbb{R}^{d}:|x|\leq N\}$. Since for each
\[
\mathbb{\bar{F}}_{{\bar X}}[\mathbf{1}_{K^{c}}]=\mathbb{\bar{E}}[\mathbf{1}%
_{\{|{\bar X}|>N\}}]\leq \frac{1}{N^{l}}\mathbb{\bar{E}}[|{\bar
X}|^{l}]<\varepsilon.
\]
It follows that the sublinear distribution $\mathbb{\bar{F}}_{\bar{X}}$ of
$\bar{X}$ is tight. The proof is complete.
\end{proof}

\begin{theorem}
\label{Theorem2}Let $(\Omega,\mathcal{H},\mathbb{E})$ be a nonlinear
expectation space and let $\{X_{\alpha}\}_{\alpha \in \mathcal{A}}$ be a family
of $\mathbb{R}^{d}$-valued random variables such that%
\[
\mathbb{E}[\varphi(X_{\alpha})]-\mathbb{E}[\psi(X_{\alpha})]\leq
\mathbb{\bar{E}}[\varphi(\bar{X})-\psi(\bar{X})],\  \  \forall \alpha
\in \mathcal{A}%
\]
where $\bar{X}$ is as in the above lemma. Then the family of distributions $\{
\mathbb{F}_{X_{\alpha}}[\varphi]\}_{\alpha \in \mathcal{A}}$ on $(\mathbb{R}%
^{d},C_{b}(\mathbb{R}^{d}))$ is weakly compact.
\end{theorem}

\begin{proof}
We denote by $L_{l}(\mathbb{R}^{d})$ the space of all $L_{0}(\mathbb{R}^{d}%
)$-measurable functions such that $|\psi(x)|\leq c(1+|x|^{l})$. It is clear
that, for each $n=1,2,\cdots$, $\mathbb{F}_{\xi_{n}}[\varphi]:=\mathbb{E}%
[\varphi(\xi_{n})]:L_{l}(\mathbb{R}^{d})\mapsto \mathbb{R}$ forms a nonlinear
distribution on $(\mathbb{R}^{d},L_{l}(\mathbb{R}^{n}))$. Moreover
\[
\mathbb{F}_{X_{\alpha}}[\varphi]-\mathbb{F}_{X_{\alpha}}[\psi]\leq
\mathbb{\bar{F}}_{\bar{X}}[\varphi-\psi],\  \  \varphi,\psi \in C_{b}%
(\mathbb{R}^{d}),\  \alpha \in \mathcal{A}.
\]
Since by the above lemma $\mathbb{\bar{F}}_{\bar{X}}$ is tight. It follows
from Theorem \ref{Theorem1} and the above lemma that $\{ \mathbb{F}%
_{X_{\alpha}}[\varphi]\}_{\alpha \in \mathcal{A}}$ is weakly compact. The proof
is complete.
\end{proof}

Let $S$ be a complete separable metric space. We denote by $C_{b}(S)$, the
linear space of all real valued bounded and continuous functions on $S$. We
can also apply Theorem \ref{Theorem1} to random variables on a nonlinear
expectation space with values on $S$. A typical situation is $S=C([0,\infty
),\mathbb{R}^{d})$.

\begin{definition}
Let $(\Omega,\mathcal{H},\mathbb{E})$ be a nonlinear expectation space. A
mapping $\xi:\Omega \mapsto S$ is said to be an $S$-valued random variable if
$f(\xi)\in \mathcal{H}$ for each $f\in C_{b}(S\mathcal{)}$.
\end{definition}

\begin{definition}
Let $(\Omega,\mathcal{H},\mathbb{E})$ be a nonlinear expectation space and
$\left \{  \xi_{i}\right \}  _{i=1}^{\infty}$ be a sequence of $S$-valued random
variables in this space. $\left \{  \xi_{i}\right \}  _{i=1}^{\infty}$ is said
to be convergent in distribution, or in law if for each $\varphi \in C_{b}(S)$,
the sequence $\{ \mathbb{E}[\varphi(\xi_{i})]\}_{i=1}^{\infty}$ is Cauchy. We
define
\[
\mathbb{F}[\varphi]=\lim_{n\rightarrow \infty}\mathbb{E}[\varphi(\xi
_{i})],\  \  \varphi \in C_{b}(S)\text{.}%
\]
$\mathbb{F}:C_{b}(S)\mapsto \mathbb{R}$ forms a nonlinear expectation on
$(S,C_{b}(S))$.
\end{definition}

\section{Central limit theorem under sublinear expectations}

Let $(\Omega,\mathcal{H},\mathbb{\hat{E}})$ be a sublinear
expectation space. We consider an i.i.d. sequence $\left \{
X_{i}\right \} _{i=1}^{\infty}$
valued in $\mathbb{R}^{d}$ such that $\mathbb{\hat{E}}[X_{1}]=\mathbb{\hat{E}%
}[-X_{1}]=0$, and $\mathbb{\hat{E}}[|X_{1}|^{2+\beta}]<\infty$ for some
constant $\beta>0$. The following result was obtained in [Peng2007] (see also
[Peng2008, 2010]) using a very deep result of partial differential equation.
We shall give a short and probabilistic proof.

\begin{theorem}
Let $S_{n}=X_{1}+\cdots+X_{n}$. The sequence $\left \{  n^{-1/2}S_{n}\right \}
_{n=1}^{\infty}$ converges in distribution to $\xi$ satisfying
\[
\mathbb{E}[\xi]=0,\  \  \mathbb{E}[\left \langle A\xi,\xi \right \rangle ]=\frac
{1}{2}G(A):=\mathbb{E}[\left \langle AX_{1},X_{1}\right \rangle ],\  \ A\in
\mathbb{S}(d).
\]
and
\[
a\xi+b\bar{\xi}\sim \sqrt{a^{2}+b^{2}}\xi,\  \  \ a,b\geq0.
\]
Namely, the limit $\xi$ is $G$-normal.
\end{theorem}

\begin{proof}
We set%
\[
\mathbb{\bar{F}}[\varphi]=\sup_{n}\mathbb{\hat{E}}\left[  \varphi
(n^{-1/2}S_{n})\right]  ,\  \  \varphi \in C_{b}(\mathbb{R}^{d}).
\]
$\mathbb{\bar{F}}$ is also a sublinear expectation on $(\mathbb{R}^{d}%
,C_{b}(\mathbb{R}^{d}))$ satisfying
\[
\mathbb{\hat{E}}\left[  \varphi(n^{-1/2}S_{n})\right]  -\mathbb{\hat{E}%
}\left[  \psi(n^{-1/2}S_{n})\right]  \leq \mathbb{\bar{F}}[\varphi
-\psi],\  \varphi,\psi \in C_{b}(\mathbb{R}^{d}).\
\]
Since $\mathbb{\hat{E}}\left[  \left \vert n^{-1/2}S_{n}\right \vert
^{2}\right]  =\mathbb{\hat{E}}[|X_{1}|^{2}]$, it is clear that, for
$\varphi=x^{2}$,
\[
\mathbb{\bar{F}}[|x|^{2}]=\mathbb{\hat{E}}[|X_{1}|^{2}]<\infty.
\]%
\[
.
\]
By Theorem \ref{Theorem1}, there exists a subsequence $\left \{  n_{i}%
^{-1/2}S_{n_{i}}\right \}  _{i=1}^{\infty}$ of $\left \{  n^{-1/2}S_{n}\right \}
_{n=1}^{\infty}$ which converges in law to some $\xi_{1}$ in $(\Omega
,\mathcal{H},\mathbb{E}\mathbb{)}$. We also have %
\[
\mathbb{E}[\xi_{1}]=\mathbb{E}[-\xi_{1}]=0,\  \  \frac{1}{2}\mathbb{E}%
[\left \langle A\xi_{1},\xi_{1}\right \rangle ]=\frac{1}{2}\mathbb{E}%
[\left \langle AX_{1},X_{1}\right \rangle ]=G(A).
\]
We will prove that $\xi_{1}$ is $G$-normal distributed for the above $G$.
Since this distribution is uniquely defined, thus we can prove that, in fact,
the sequence $\{ \mathbb{\hat{E}}\left[  \varphi(n^{-1/2}S_{n})\right]
\}_{n=1}^{\infty}$ converges in law to the $G$-normal distribution.

For the above convergent subsequence $\left \{  n_{i}^{-1/2}S_{n_{i}}\right \}
_{i=1}^{\infty}$, it is clear that for an arbitrarily increasing integers of
$\{ \bar{n}_{i}\}_{i=1}^{\infty}$ such that $|\bar{n}_{i}-n_{i}|\leq1$, both
$\left \{  n_{i}^{-1/2}S_{n_{i}}\right \}  _{i=1}^{\infty}$ and $\left \{
\bar{n}_{i}^{-1/2}S_{\bar{n}_{i}}\right \}  _{i=1}^{\infty}$ converges in law
to a same limit. Thus, without loss of generality, we can assume that $n_{i}$,
$i=1,2,\cdots$, are all even numbers and decompose $n_{i}^{-1/2}S_{n_{i}}$
into two parts:
\[
n_{i}^{-1/2}S_{n_{i}}=\frac{1}{\sqrt{2}}(n_{i}/2)^{-1/2}S_{\frac{n_{i}}{2}%
}+\frac{1}{\sqrt{2}}(n_{i}/2)^{-1/2}(S_{n_{i}}-S_{\frac{n_{i}}{2}})
\]
We then use the same argument to each part and prove that there
exists a subsequence of
$\{\frac{1}{\sqrt{2}}(n_{i}/2)^{-1/2}S_{\frac{n_{i}}{2}}\}_{i=1}^{\infty}$
(resp.
$\{\frac{1}{\sqrt{2}}(n_{i}/2)^{-1/2}(S_{n^{i}}-S_{\frac{n_{i}}{2}})\}_{i=1}^{\infty}$)
converging in law to $\xi_{1/2}$ such that $\xi_{1}-\xi_{1/2}$ is
identically distributed with respect to $\xi_{1/2}$ as well as
independent of $\xi_{1/2}$. We set $\xi_{0}=0$. It is easy to check
that
\begin{equation}
\mathbb{E}[\xi_{t}]=\mathbb{E}[-\xi_{t}]=0,\  \  \frac{1}{2}\mathbb{E}%
[\left \langle A\xi_{t},\xi_{t}\right \rangle ]=G(A)t,\ A\in \mathbb{S}%
(d),\  \label{3.1}%
\end{equation}
for $t=0,\frac{1}{2},1$.

We can continue this procedure to prove that there exist random variables
$\xi_{t}$, $t=i/4$, $i=0,1,\cdots,4$, such that (\ref{3.1}) holds for each $t$
and each increment $\xi_{t}-\xi_{s}$ is identically distributed as $\xi_{t-s}$
and independent of $\xi_{t_{1}}$, $\xi_{t_{2}}$, for $t,s,t_{1},t_{2}=i/4$,
$i=0,1,\cdots,4$ and $t\geq s\geq t_{i}$, $i=1,2$. Here, using the technique
of the product of expectation spaces (see [Peng2010]), we always use the same
sublinear expectation space $(\Omega,\mathcal{H},\mathbb{E})$ like in the
classical situation.

We proceeding in this way to prove that we can find $\xi_{t}\in \mathcal{H}$,
$t\in \mathbf{Q}[0,1]$ ($\mathbf{Q}[0,1]$ denotes the set of all rational
numbers in $[0,1]$) such that $\xi_{t+h}-\xi_{t}$ is identically distributed
with respect to $\xi_{h}$ and is independent of $\xi_{t_{1}}$, $\xi_{t_{2}}%
$,$\cdots,\xi_{t_{k}}$, for each $t$, $s$, $h$, $t_{i}$ in $[0,1]$ such that
$t\geq t_{i}$, $i=1,2,\cdots k$. $t+h\in \lbrack0,1]$. We also have (\ref{3.1})
for $t\in \mathbf{Q}[0,1]$.

It remains to prove that $\xi_{1}$ is $G$-normally distributed. The
method is very similar to that of Theorem 1.6 in [Peng2010, ChIII].
We only give a sketch of the proof.

We define, for a bounded and Lipschitz function $\varphi$,%
\[
u(t,x)=\mathbb{E}[\varphi(x+\xi_{t})],\  \ x\in \mathbb{R}^{d},\  \ t\in
\mathbf{Q}[0,1].
\]
It is easy to check that $u(t,x)$ is a Lipschitz function of $x$ with the same
Lipschitz constant as $\varphi$ and, for each $t$, $s$ $\in \mathbf{Q}[0,1]$
such that $t+s\leq1$,
\begin{equation}
u(t+s,x)=\mathbb{E}[u(t,x+\xi_{s})],\  \ x\in \mathbb{R}^{d}.\  \  \label{3.2}%
\end{equation}
From which we can prove that%
\begin{align*}
|u(t+s,x)-u(t,x)|  &  =\mathbb{E}[|u(t,x+\xi_{s})-u(t,x)|]\\
&  \leq C\mathbb{E}[|\xi_{s}|]\leq \bar{C}\sqrt{s}.
\end{align*}
We then continuously extend the domain $u(t,x)$ from $\mathbf{Q}%
[0,1]\times \mathbb{R}^{d}$ to $[0,1]\times \mathbb{R}^{d}$. Relation
(\ref{3.2}) still holds for $t\in \lbrack0,1)$, $s\in \mathbf{Q}[0,1]$ such that
$t+s\leq1$.

For each $\bar{x}\in \mathbb{R}^{d}$, $\bar{t}\in \lbrack0,1)$, if $\psi$ is a
smooth function on $[0,1]\times \mathbb{R}^{d}$ such that $\psi \geq u$ and
$\psi(\bar{t},\bar{x})=u(\bar{t},\bar{x})$. From (\ref{3.2}) we have, for a
sufficiently small rational number $s>0$,%
\[
\partial \psi(\bar{t},\bar{x})s-\mathbb{E[}\left \langle D\psi(\bar{t},\bar
{x}),\xi_{s}\right \rangle +\frac{1}{2}\left \langle D^{2}\psi(\bar{t},\bar
{x})\xi_{s},\xi_{s}\right \rangle ]+o(s)\leq0\text{.}%
\]
Thus
\[
\partial_{t}\psi(\bar{t},\bar{x})-G(D^{2}\psi(\bar{t},\bar{x}))\leq
0,\  \ u|_{t=0}=\varphi.
\]
This implies $u$ is a viscosity subsolution of the PDE%
\[
\partial_{t}u-G(D^{2}u)=0.
\]
Similarly we can also prove that $u$ is a viscosity supersolution.
Thus $\xi_{1}$ is $G$-normally distributed. The proof is complete.
\end{proof}

\begin{remark}
The above probabilistic proof of central limit theorem can be also applied to
prove the following more general situation. We assume that $\{(X_{i}%
,Y_{i})\}_{i=1}^{\infty}$ is an i.i.d sequence in the following sense: for
each $i$, $(X_{i+1},Y_{i+1})$ is identically distributed with respect to
$(X_{1},Y_{1})$ and is independent of $(X_{1},Y_{1})$, $\cdots,(X_{i},Y_{i})$,
we also assume that $\mathbb{E}[X_{1}]=\mathbb{E}[-X_{1}]=0$, and $|X_{i}%
|^{3}$, $|Y_{i}|^{2}\in \mathcal{H}$. Then we have
\[
\sum_{i=1}^{n}(\frac{X_{i}}{\sqrt{n}}+\frac{Y_{i}}{n})\overset{\text{in law}%
}{\rightarrow}(\xi,\eta)
\]
where $(\xi,\eta)$ is $G$-distributed with
\[
G(p,A):=\mathbb{E}[\left \langle p,Y_{1}\right \rangle +\frac{1}{2}\left \langle
AX_{1},X_{1}\right \rangle ].
\]
This implies that the function $u(t,x):=\mathbb{E}[\varphi(x+\sqrt{t}\xi
+t\eta)]$ is the unique solution of the PDE
\[
\partial_{t}u-G(Du,D^{2}u)=0,\  \ u|_{t=0}=\varphi.
\]

\end{remark}

\begin{remark}
The above result provides a new method to the proof of the existence of
solutions for a type of degenerate fully nonlinear and possibly degenerate
PDEs. It also indicates a numerical method to solve such type of PDE. This
approach is quite different from the well-known approach of Perron's method
for the existence. Perron's method is note constructive.
\end{remark}

\begin{remark}
The above method can be also applied to treat the solution existence of the
following type of SDE
\[
dX_{t}=b(X_{t})dt+\beta(X_{t})d\left \langle B\right \rangle _{t}+\sigma
(X_{t})dB_{t}%
\]
for the case when $b$, $\beta$ and $\sigma$ are only continuous.
\end{remark}

\end{document}